\documentclass[reqno]{amsart}
\usepackage[foot]{amsaddr}
\usepackage{amssymb,amsmath,amsfonts,mathptmx,cite,mathrsfs,url}
\usepackage[mathscr]{eucal}

\DeclareSymbolFont{rsfscript}{OMS}{rsfs}{m}{n}
\DeclareSymbolFontAlphabet{\mathrsfs}{rsfscript}

\newtheorem{theorem}{Theorem}[section]
\newtheorem{proposition}[theorem]{Proposition}
\newtheorem{lemma}[theorem]{Lemma}
\newtheorem{fact}[theorem]{Fact}
\newtheorem{corollary}[theorem]{Corollary}

\theoremstyle{remark}
\newtheorem{remark}{Remark}

\newcommand{\ais}{ai-semi\-ring}

\newcommand{\mA}{\mathcal{A}}

\def\fb{finitely based}

\newcommand{\sgp}{semi\-group}

\newcommand{\qid}{quasi-iden\-ti\-ty}

\DeclareMathOperator{\var}{var}
\DeclareMathOperator{\qvar}{qvar}

\DeclareMathOperator{\ps}{ps}

\numberwithin{equation}{section}

\begin{document}

\title[Lattices of group quasivarieties and additively idempotent semiring varieties]{Embedding lattices of quasivarieties of periodic groups into lattices of additively idempotent semiring varieties: An algebraic proof}

\author[Miaomiao Ren]{Miaomiao Ren}
\address[Miaomiao Ren]{{\normalfont School of Mathematics, Northwest University, Xi'an, Shaanxi, 710127, P.R. China}}
\email{miaomiaoren@yeah.net}
\author[Xianzhong Zhao]{Xianzhong Zhao}
\address[Xianzhong Zhao, M. V. Volkov]{{\normalfont School of Mathematics and Data Science, Shaanxi University of Science and Technology, Xi'an, Shaanxi, 710021, P. R. China}}
\email{zhaoxz@nwu.edu.cn}
\author[M. V. Volkov]{Mikhail V. Volkov}
\email{mishavolkov@gmail.com}

\thanks{Miaomiao Ren is supported by the National Natural Science Foundation of China (Grant Nos. 12371024, 12571020).
Xianzhong Zhao is supported by the National Natural Science Foundation of China (Grant No. 12571020).}

\begin{abstract}
A general result by Jackson (Flat algebras and the translation of universal Horn logic to equational logic, J. Symb. Log. 73(1) (2008) 90--128) implies that the lattice of all quasivarieties of groups of exponent dividing $n$ embeds into the lattice $L(\mathbf{Sr}_n)$ of all varieties of additively idempotent semirings whose multiplicative semigroups are unions of groups of exponent dividing $n$; the image of this embedding is an interval in $L(\mathbf{Sr}_n)$. We provide a new, direct, and purely algebraic proof of these facts and present a new identity basis for the top variety of the interval. In addition, we obtain new information about the lattice $L(\mathbf{Sr}_n)$, demonstrating that the properties of the lattice for $n\ge 3$ differ drastically from those previously known when $n=1$ or $2$.
\end{abstract}



\maketitle

\section{Introduction}
\label{sec:introduction}

An \emph{additively idempotent semiring} (\ais, for short) is an algebra $(S;\,+,\cdot)$ with two binary operations, addition $+$ and multiplication $\cdot$, such that
\begin{itemize}
\item the additive reduct $(S;\,+)$ is a commutative idempotent semigroup, that is,
\[
s+(t+u)=(s+t)+u,\ \ s+t=t+s,\ \text{ and }\ s+s=s\ \text{ for all }\ s,t,u\in S;
\]
\item the multiplicative reduct $(S;\,\cdot)$ is a semigroup, that is,
\[
s(tu)=(st)u\ \text{ for all }\ s,t,u\in S;
\]
\item multiplication distributes over addition on the left and on the right, that is,
\[
s(t+u)=st+su\ \text{ and }\ (s+t)u=su+tu\ \text{ for all }\ s,t,u\in S.
\]
\end{itemize}

Ai-semirings naturally arise in various areas of mathematics (e.g., idempotent analysis, tropical geometry, and optimization) and computer science, where rings and other ``classical'' algebras fail to provide adequate approaches; see the contributions in the conference volumes \cite{Guna98,LiMa:2005,LiSe:2009} and the monograph \cite{Bis04} for diverse examples. The study of ai-semirings in their own right has developed into a rich theory, revealing many interesting phenomena.

Since \ais{}s are defined by identities, the class $\mathbf{AI}$ of all \ais{}s forms a variety, as do many subclasses of $\mathbf{AI}$ considered in the literature. This makes it promising to study \ais{}s within the framework of the theory of varieties, and indeed, the varietal approach to \ais{}s has proved fruitful. The main topics of research in this direction have been the finite axiomatizability question (also known as the Finite Basis Problem) for \ais{}s and the description of certain sublattices in the lattice of all subvarieties of $\mathbf{AI}$.

The present paper focuses on the lattice $L(\mathbf{Sr}_n)$ of all subvarieties of the variety $\mathbf{Sr}_n$ consisting of all \ais{}s whose multiplicative semigroups are unions of groups of exponent dividing a fixed positive integer $n$. The lattices $L(\mathbf{Sr}_1)$ and $L(\mathbf{Sr}_2)$ were described in \cite{GPZ05,Pas05} and \cite{RZW17}, respectively, but the lattices $L(\mathbf{Sr}_n)$ with $n\ge 3$ have remained largely unexplored. Our main result, Theorem~\ref{thm:main}, identifies in $L(\mathbf{Sr}_n)$ an interval isomorphic to the lattice of all \emph{subquasivarieties} of the group variety $\mathbf{G}_n$ consisting of all groups of exponent dividing $n$. This, combined with known results about group quasivarieties, allows us to resolve several previously open questions concerning $L(\mathbf{Sr}_n)$. Specifically, we derive that the lattice $L(\mathbf{Sr}_n)$ has the cardinality of the continuum for $n\ge 3$ and is non-modular whenever $n$ is divisible by either the square of an odd prime or by $4p$, where $p$ is any prime.

The fact that the lattice of all subquasivarieties of $\mathbf{G}_n$ embeds into the lattice $L(\mathbf{Sr}_n)$ can be derived from a general result by Jackson \cite{Jackson08}, who discovered a method for translating universal Horn logic into equational logic. Jackson's powerful approach has found numerous applications in the study of varieties of algebras of various types, including \ais{}s; see \cite{Jackson19,JackRenZhao22,RJZL23} for semiring-specific examples. In particular, Theorem 2.1 in \cite{RJZL23}, which specializes \cite[Corollary 5.4, Theorem 5.12]{Jackson08}, captures a significant part of our Theorem~\ref{thm:main}. Nonetheless, our contribution offers two novel aspects. First, we provide direct and purely algebraic proofs that avoid the syntactic transformations of logical formulas which underlie Jackson's arguments. Second, we present a simple identity basis for the top variety in the image of the embedding under study. We discuss the relations between our approach and Jackson's in more detail in the last section of this paper.

The paper traces its origins to the third-named author's visit to Northwest University, Xi'an, during the fall of 2016. The authors' arXiv preprint \cite{RZV22} includes some results that are sharpened and further developed in the present paper.

\section{Preliminaries}
\label{sec:prelim}

We will require several concepts from universal algebra, as well as from group, semigroup, and lattice theory. Most of these are standard and can be found in the textbooks \cite{BuSa81,Clifford&Preston:1961,Gr:2011,Hall:1959,Howie:1995}. For the reader's convenience, we briefly review some basics here, setting up our notation along the way. We also recall a few known facts about \ais{}s that will be needed later.

\subsection{Identities and varieties}
We fix a countably infinite set $X$, consisting of letters such as $x,y,z$, with and without indices. The elements of $X$ are called \emph{variables}. A \emph{word} is a finite sequence of variables. (\emph{Semiring}) \emph{polynomials} are defined inductively starting from words: every word is a polynomial, and if $p_1$ and $p_2$ are polynomials, then so are the expressions $p_1+p_2$ and $p_1p_2$.

Any map $\varphi\colon X\to S$, where $(S;\,+,\cdot)$ is an \ais, is called a \emph{substitution}. The \emph{value} $p\varphi$ of a polynomial $p$ under $\varphi$ is defined inductively as follows: if $p=x_1\cdots x_k$ is a word, where $x_1,\dots,x_k$ are variables, then $p\varphi:=x_1\varphi\cdots x_k\varphi$; if $p=p_1+p_2$ or $p=p_1p_2$ for some polynomials $p_1$ and $p_2$ whose values under $\varphi$ have already been defined, then $p\varphi:=p_1\varphi+p_2\varphi$ or $p\varphi:=p_1\varphi\cdot p_2\varphi$, respectively.

A (\emph{semiring}) \emph{identity} is an expression of the form $p=p'$, where $p$ and $p'$ are polynomials. An \ais{} $(S;\,+,\cdot)$ \emph{satisfies} $p=p'$ (or $p=p'$ \emph{holds} in $(S;\,+,\cdot)$) if $p\varphi=p'\varphi$ for every substitution $\varphi\colon X\to S$. That is, each substitution of elements from $S$ for the variables occurring in $p$ or $p'$ yields equal values for the polynomials.

An identity is said to be satisfied in a class of \ais{}s if it holds in every \ais{} in this class. The class of all \ais{}s that satisfy all identities from a given set $\Sigma$ is called the \emph{variety axiomatized by $\Sigma$}.

The least variety containing a class $\mathbf{K}$ of \ais{}s is said to be \emph{generated} by $\mathbf{K}$ and is denoted by $\var\mathbf{K}$. In terms of identities, the variety $\var\mathbf{K}$ is axiomatized by all identities that hold in $\mathbf{K}$; alternatively, $\var\mathbf{K}$ can be characterized as the class of homomorphic images of subsemirings of direct products of \ais{}s from $\mathbf{K}$ (the HSP Theorem; see \cite[Theorem II.9.5]{BuSa81}).

The varieties of \ais{}s form a complete lattice under class inclusion. For a family $\{\mathbf{V}_i\}_{i\in I}$ of varieties, their lattice meet is the intersection $\bigcap_{i\in I}\mathbf{V}_i$, and their lattice join is the variety generated by the union $\bigcup_{i\in I}\mathbf{V}_i$. The subvarieties of a given variety $\mathbf{V}$ form a complete sublattice, denoted by $L(\mathbf{V})$.

We will also encounter some group and semigroup varieties; these are classes of groups and semigroups axiomatized by sets of group and semigroup identities, respectively.

\subsection{Quasi-identities and quasivarieties}
A (\emph{semigroup}) \emph{quasi-identity} is an expression of the form
\begin{equation}
\label{generic}
{u}_1=  {v}_1\ \&\ {u}_2=  {v}_2\ \&\ \cdots\ \&\ {u}_m=  {v}_m\longrightarrow {u}=  {v},
\end{equation}
where ${u}_1,{v}_1,{u}_2,{v}_2,\dots,{u}_m,{v}_m,{u},{v}$ are words. This definition allows $m=0$, in which case \eqref{generic} reduces to the identity ${u}={v}$.

A \sgp\ $(S;\,\cdot)$ \emph{satisfies} \qid\ \eqref{generic} if each substitution $\varphi\colon X\to S$ yields ${u}\varphi={v}\varphi$, provided that ${u}_1\varphi={v}_1\varphi$, ${u}_2\varphi={v}_2\varphi$, \dots, ${u}_m\varphi={v}_m\varphi$.

The class of all semigroups that satisfy all quasi-identities from a given set $\Sigma$ is called the \emph{quasivariety axiomatized by $\Sigma$}. Quasivarieties are characterized as the classes that contain the one-element semigroup and are closed under taking isomorphic copies of subsemigroups as well as under forming direct products and ultraproducts; see \cite[Theorem V.2.25]{BuSa81}. The least quasivariety containing a class $\mathbf{K}$ of semigroups is said to be \emph{generated} by $\mathbf{K}$ and is denoted by $\qvar\mathbf{K}$.

The quasivarieties of semigroups form a complete lattice under class inclusion, with meet and join operations defined  analogously to those in the lattice of varieties. For a quasivariety $\mathbf{Q}$, we denote the lattice of all its subquasivarieties by $L_q(\mathbf{Q})$.

\subsection{Clifford semigroups}
An \emph{idempotent} of a semigroup $(S;\,\cdot)$ is an element $e\in S$ such that $e^2=e$. We denote the set of all idempotents of $(S;\,\cdot)$ by $E(S)$.

We say that $(S;\,\cdot)$ has \emph{commuting idempotents} if $ef=fe$ for all $e,f\in E(S)$. A semigroup with commuting idempotents that is a union of its subgroups is called a \emph{Clifford} semigroup.

We will need two well-known properties of Clifford semigroups.
\begin{lemma}[{\!\mdseries\cite[Lemma 4.8]{Clifford&Preston:1961}}]\label{lem:clifford}
If $(S;\,\cdot)$ is a Clifford semigroup, then $es=se$ for all $e\in E(S)$ and $s\in S$.
\end{lemma}

By a \emph{semilattice}, we mean a semigroup consisting of commuting idempotents. Let $(Y;\,\cdot)$ be a semilattice and $\{G_\alpha\}_{\alpha\in Y}$ a family of disjoint groups\footnote{Here and throughout, we denote a group and its underlying set by the same symbol, using $\cdot$ as the default notation for group multiplication.} indexed by the elements of $Y$. A semigroup $(S;\cdot)$ is said to be a \emph{semilattice of groups} $G_\alpha$, $\alpha\in Y$, if:
\begin{itemize}
  \item $S=\bigcup_{\alpha\in Y}G_\alpha$;
  \item each $G_\alpha$ is a subgroup of $(S;\cdot)$;
  \item for every $\alpha,\beta\in Y$ and every $a\in G_\alpha$, $b\in G_\beta$, the product $ab$ belongs to $G_{\alpha\beta}$.
\end{itemize}

\begin{lemma}[{\!\mdseries\cite[part of Theorem 4.11]{Clifford&Preston:1961}}]\label{lem:clifthm}
A semigroup is Clifford if and only if it is a semilattice of groups.
\end{lemma}

\subsection{Flat extensions of groups}
Let $G$ be a group and let 0 be a symbol not occurring in $G$. The \emph{flat extension} of $G$ is the algebra $G^\flat:=(G\cup\{0\};\,+,\cdot)$ where addition is defined for $a,b\in G\cup\{0\}$ by
\begin{equation}
\label{eq:flat}
a+b:=\begin{cases}
0&\text{if } a\ne b,\\
a&\text{if } a=b,
\end{cases}
\end{equation}
and multiplication extends that of $G$ by preserving the original multiplication on elements of $G$  and defining all products involving 0 to be 0. It is known (and easy to verify) that the flat extension of any group is an \ais{}.

We call an algebra \emph{simple} if its only congruences are the diagonal (the equality relation) and the universal relation. Flat extensions of groups are known to be simple \ais{}s. As we were unable to find a convenient reference, we include a short proof of this fact.
\begin{lemma}\label{lem:flatsimple}
The flat extension of any group is a simple \ais.
\end{lemma}

\begin{proof}
Let $G$ be a group and $\rho$ a non-diagonal congruence on $G^\flat$. Take $a\ne b$ such that $a\mathrel{\rho}b$. We may assume $a\ne 0$. Since $\rho$ respects addition, adding $a$ to both sides of $a\mathrel{\rho}b$ yields $(a+a)\mathrel{\rho}(a+b)$. By \eqref{eq:flat}, we have $a+a=a$ and $a+b=0$, so $a\mathrel{\rho}0$. As $G$ is a group, for any $c\in G$, there exists $d\in G$ with $c=ad$. Since $\rho$ respects multiplication, multiplying both sides of $a\mathrel{\rho}0$ by $d$ yields $c=ad\mathrel{\rho}0\cdot d=0$. Hence, every $c\in G$ is $\rho$-related to 0, so $\rho$ is the universal relation.
\end{proof}

\subsection{Notational conventions}
Throughout this paper, $n$ is a positive integer.

Some of our computations involve expressions of the form $uv^{n-1}$ where $u$ and $v$ are either words or elements of an \ais. To ensure such expressions remain meaningful even when $n=1$, we adopt the convention $uv^0:=u$.

We will often work with expressions of the form $w+w^n$, where $w$ is either a word or an element of an \ais{}. To streamline notation, we write $\mathrm{M}(w)$ to denote $w+w^n$, regardless of the interpretation of $w$. In particular, $\mathrm{M}(x)$ stands for the polynomial $x+x^n$.

\section{{Some properties of the variety $\mathbf{Sr}_n$}{}}
\label{sec:srn}
A group $G$ with identity element $e$ is said to have exponent $n$ if $n$ is the smallest positive integer such that $g^n=e$ for every $g\in G$. For any fixed $n$, the class $\mathbf{G}_n$ of all groups of exponent dividing $n$ is a semigroup variety that can be axiomatized by the identities
\[
x^ny=yx^n=y.
\]
The class $\mathbf{Sr}_n$ of all \ais{}s whose multiplicative reducts are unions of groups from $\mathbf{G}_n$ is a variety that can be axiomatized by the single identity
\begin{equation}
\label{eq:srn} x=x^{n+1}.
\end{equation}

Structural properties of \ais{}s in the variety $\mathbf{Sr}_n$ were studied in \cite{RZS16}\footnote{To avoid possible confusion, notice that the variety denoted by $\mathbf{Sr}_n$ here corresponds to $\mathbf{Sr}(n+1,1)$ in the notation of \cite{RZS16}.}. We will need two properties concerning idempotents in such \ais{}s.

\begin{lemma}[{\!\mdseries\cite[Lemma 2.1]{RZS16}}]\label{lem:e(s)}
If $(S;\,+,\cdot)$ is an \ais{} in the variety $\mathbf{Sr}_n$, then $E(S)=\{s^{n} \mid s \in S\}$ and $(E(S);\,+,\cdot)$ is a subsemiring of $(S;\,+,\cdot)$.
\end{lemma}

The next fact follows from \cite[Theorem 2.5]{RZS16}.
\begin{lemma}\label{lem:cryptic}
The variety $\mathbf{Sr}_n$ satisfies the identity
\begin{equation}
\label{eq:cryptic} (xy)^n=  x^{n}y^n.
\end{equation}
\end{lemma}

The following lemma records several useful identities in the variety $\mathbf{Sr}_n$ involving the polynomial $\mathrm{M}(x):=x + x^n$.
\begin{lemma}\label{lem:m}
The variety $\mathbf{Sr}_n$ satisfies the identities
\begin{equation}\label{21701}
\mathrm{M}(x)=x\mathrm{M}(x)=\mathrm{M}(x)x=(\mathrm{M}(x))^2.
\end{equation}
\end{lemma}

\begin{proof}
Let $e:=x^n$ so that $\mathrm{M}(x)=x+e$. By \eqref{eq:srn}, we have $ex=xe=x^{n+1}=x$, $e^2=e$, and $x+e=(x+e)^{n+1}$. Expanding the binomial $(x+e)^{n+1}$ and using $ex=xe=x$ and $e^2=e$, along with commutativity and idempotency of addition, we get
\[
(x+e)^{n+1}=x^{n+1}+x^ne+x^{n-1}e+\cdots+xe+e\stackrel{\eqref{eq:srn}}=x+x^2+\cdots+x^n.
\]
Hence, we have $\mathrm{M}(x) = x+x^2+\cdots+x^n$. Using this expression for $\mathrm{M}(x)$, we deduce
\begin{gather*}
\mathrm{M}(x)x=x\mathrm{M}(x)=x^2+x^3+\cdots+x^{n+1}\stackrel{\eqref{eq:srn}}=x+x^2+\cdots+x^n=\mathrm{M}(x)\\
\intertext{and}
\mathrm{M}(x)^2=\mathrm{M}(x)(x+x^2+\cdots+x^n)=\mathrm{M}(x)x+\mathrm{M}(x)x^2+\cdots+\mathrm{M}(x)x^n=\mathrm{M}(x),
\end{gather*}
thus establishing all identities in \eqref{21701}.
\end{proof}

For a binary relation $\theta$ on a set $S$ and a subset $T\subseteq S$, let $\theta|_T$ stand for the \emph{restriction} of $\theta$ to $T$, that is, $\theta|_T:=\theta\cap(T\times T)$.
\begin{lemma}\label{lem:restriction}
Let $\theta$ be a congruence on an \ais{} $(S;\,+,\cdot)$ from the variety $\mathbf{Sr}_n$. If $\theta|_{E(S)}$ is the diagonal of $E(S)$, then $\theta$ is the diagonal of $S$.
\end{lemma}

\begin{proof}
We aim to show that $a,b\in S$ with $a\mathrel{\theta}b$ must be equal. If $a\mathrel{\theta}b$, then $a^n\mathrel{\theta}b^n$. By Lemma~\ref{lem:e(s)}, $a^n,b^n\in E(S)$, and since $\theta|_{E(S)}$ is the diagonal of $E(S)$, we have $a^n=b^n$.

Let $c:=ab^{n-1}$. We have $c^n=(ab^{n-1})^n\stackrel{\eqref{eq:cryptic}}=a^nb^{(n-1)n}=(b^n)^n=b^n$. Since $a\mathrel{\theta}b$, we also have $c=ab^{n-1}\mathrel{\theta}b^n=c^n$. Adding $c^n$ to both sides of $c\mathrel{\theta}c^n$, we obtain $(c+c^n)\mathrel{\theta}c^n$. By Lemma~\ref{lem:m}, the element $\mathrm{M}(c)=c+c^n$ is an idempotent, and so is $c^n$ by Lemma~\ref{lem:e(s)}. Therefore, $\mathrm{M}(c)=c^n$ since $\theta|_{E(S)}$ is the diagonal of $E(S)$. Using $\mathrm{M}(c)=c^n$, we now compute:
\[
c\stackrel{\eqref{eq:srn}}=cc^n=c\mathrm{M}(c)\stackrel{\eqref{21701}}=\mathrm{M}(c)=c^n.
\]
Thus, $ab^{n-1}=c=c^n=b^n$. Multiplying both sides of $ab^{n-1}=b^n$ by $b$ and using the already established equality $a^n=b^n$, we obtain $a\stackrel{\eqref{eq:srn}}=aa^n=ab^n=b^{n+1}\stackrel{\eqref{eq:srn}}=b$.
\end{proof}

\section{{Variety $\mathbf{M}_n$ and its subdirectly irreducible members}{}}
\label{sec:mn}

Let $\mathbf{M}_n$ be the subvariety of the variety $\mathbf{Sr}_n$ defined within $\mathbf{Sr}_n$ by the identity
\begin{equation}
\label{new3}
x^n+y^n=x^ny^n.
\end{equation}
To clarify the meaning of \eqref{new3}, recall that $E(S)=\{s^{n} \mid s \in S\}$ for every \ais{} $(S;\,+,\cdot)$ in $\mathbf{Sr}_n$ (Lemma~\ref{lem:e(s)}). Hence, an \ais{} $(S;\,+,\cdot)$ satisfies \eqref{new3} if and only if the subsemiring $(E(S);\,+,\cdot)$ satisfies the identity
\begin{equation}\label{eq:dupl}
x+y=xy.
\end{equation}
Identity \eqref{eq:dupl} means that addition and multiplication in $(E(S);\,+,\cdot)$ coincide. Thus, the \ais{} $(E(S);\,+,\cdot)$ is a \emph{duplicated semilattice}\footnote{We prefer this name over the somewhat oxymoronic term \emph{monobisemilattice} used, e.g., in \cite{PZ00}.}, that is, a semilattice, where a single operation is treated as both addition and multiplication. Thus, $\mathbf{M}_n$ is the variety of all \ais{}s from $\mathbf{Sr}_n$ whose subsemirings of idempotents are duplicated semilattices. Note that the variety of all duplicated semilattices is nothing but $\mathbf{M}_1$ since for $n=1$ identity \eqref{eq:srn} is a consequence of \eqref{new3}.

In this section, we characterize the subdirectly irreducible members of $\mathbf{M}_n$.

We begin by recording two useful identities that hold in $\mathbf{M}_n$.
\begin{lemma}\label{lem:3iden}
Each of the following identities holds in the variety $\mathbf{M}_n$.
\begin{gather}
xy^n   =  y^nx;    \label{eq:center}\\
x+y    =  \mathrm{M}(xy^{n-1})y.\label{eq:summing}
\end{gather}
\end{lemma}

\begin{proof}
First, observe that the variety $\mathbf{M}_n$ satisfies the identity $x^ny^n=  y^nx^n$. Indeed, we have
\[x^ny^n \stackrel{\eqref{new3}}=  x^n+y^n =  y^n+x^n\stackrel{\eqref{new3}} =  y^nx^n.\]
In view of Lemma~\ref{lem:e(s)}, in each \ais{} in $\mathbf{Sr}_n$, and in particular in $\mathbf{M}_n$, an element is idempotent if and only if it is an $n$-th power. Therefore, the identity $x^ny^n=y^nx^n$ expresses  that the idempotents in the \ais{}s of $\mathbf{M}_n$ commute. Hence, \ais{}s in $\mathbf{M}_n$ have Clifford multiplicative semigroups. By Lemma~\ref{lem:clifford}, each idempotent in a Clifford semigroup commutes with every element. Using again that the idempotents in \ais{}s in $\mathbf{M}_n$ are exactly the $n$-th powers, we conclude that $\mathbf{M}_n$ satisfies identity \eqref{eq:center}.

To prove \eqref{eq:summing}, we first show that $\mathbf{M}_n$ satisfies the identity
\begin{equation}\label{eq:new1}
x+y    =  xy^n+x^ny.
\end{equation}
We establish this through the following computation:
\begin{align*}
x+y
&=  (x+y)^{n+1}              &&\text{by \eqref{eq:srn}}  \\
&=  (x+y)^{n+1} + xy^n+x^ny  &&\text{since the words $xy^n$ and $x^ny$ occur as summands} \\[-1ex]
&                                 &&\text{in the expansion of the binomial $(x+y)^{n+1}$}\\
&=  x+y+xy^n+x^ny            &&\text{by \eqref{eq:srn}} \\
&=  x^{n+1}+y^{n+1}+xy^n+x^ny&&\text{by \eqref{eq:srn}}\\
&=  x(x^n+y^n)+(x^n+y^n)y    &&\\
&=  xx^ny^n+x^ny^ny          &&\text{by \eqref{new3}}\\
&=  xy^n+x^ny                &&\text{by \eqref{eq:srn}.}
\end{align*}
Now we compute:
\begin{align*}
\mathrm{M}(xy^{n-1})y
&=\left(xy^{n-1}+(xy^{n-1})^n\right)y&&\\
&=\left(xy^{n-1}+x^ny^{n(n-1)}\right)y&&\text{by \eqref{eq:cryptic}}\\
&=xy^n+x^ny^{n(n-1)+1}&&\\
&=xy^n+x^ny&&\text{by \eqref{eq:srn}} \\
&=x+y&&\text{by \eqref{eq:new1},}
\end{align*}
thus establishing \eqref{eq:summing}.
\end{proof}

An \ais{} with a multiplicative zero (i.e., an absorbing element for multiplication) is called a \emph{semifield} if its nonzero elements form a group under multiplication. Flat extensions of groups are examples of semifields.
Lemma~\ref{lem:3iden} implies that these are the only semifields in the variety $\mathbf{M}_n$.

\begin{corollary}\label{cor:semifield}
Every semifield in the variety $\mathbf{M}_n$ is the flat extension of its group of nonzero elements.
\end{corollary}

\begin{proof}
Let $(S;\,+,\cdot)$ be a semifield in $\mathbf{M}_n$. Denote its multiplicative zero by 0 and let $G:=S\setminus\{0\}$. Then $G$ is a group. We will show that $(S;\,+,\cdot)$ coincides with the flat extension $G^\flat$ of $G$. By construction, multiplication in $G^\flat$ agrees with that of $(S;\,+,\cdot)$, so it remains to show that addition in $(S;\,+,\cdot)$ satisfies the rule given in \eqref{eq:flat}.

By Lemma~\ref{lem:m}, for every $s\in S$, the element $\mathrm{M}(s)=s+s^n$ is an idempotent satisfying $\mathrm{M}(s)s=\mathrm{M}(s)$. Clearly, the only idempotent in the group $G$ is its identity element, which we denote by $e$. Thus, $\mathrm{M}(s)=0$ or $\mathrm{M}(s)=e$. In the latter case, we must have $s\in G$, because if $s=0$, then so is $\mathrm{M}(s)=\mathrm{M}(s)s$. If $s\in G$, then $s=es$, and the equality $\mathrm{M}(s)=e$ implies $s=es=\mathrm{M}(s)s=\mathrm{M}(s)=e$. Therefore, for any $s\ne e$, it must be that $\mathrm{M}(s)=0$.

If $n=1$, then $e$ is the only element different from 0, and we have
\[
0+e=e+0\stackrel{\eqref{new3}}=e\cdot 0=0.
\]
If $n>1$ and $a,b\in S$ are such that $ab^{n-1}=e$, then $a,b\in G$ whence $b^n=e$ as $b^n$ is an idempotent by Lemma~\ref{lem:e(s)}. It follows that $ab^n=ae=a$. On the other hand, multiplying both sides of $ab^{n-1}=e$ on the right by $b$, we obtain $ab^n=b$. Thus, $a=b$. Therefore, for any two distinct elements $a,b\in S$, we must have $ab^{n-1}\ne e$. As established in the previous paragraph, this implies $\mathrm{M}(ab^{n-1})=0$. Then we have
\[
a+b\stackrel{\eqref{eq:summing}}= \mathrm{M}(ab^{n-1})b=0\cdot b=0,
\]
as required.
\end{proof}

Recall that, by Lemma~\ref{lem:e(s)}, the idempotents form a subsemiring in every \ais{} belonging to the variety $\mathbf{Sr}_n$. Our next lemma shows that, in an \ais{} from $\mathbf{M}_n$, congruences on the subsemiring of idempotents extend to congruences on the entire semiring.
\begin{lemma}\label{lem:extension}
Let $(S;\,+,\cdot)$ be an \ais{} in $\mathbf{M}_n$. For every congruence $\rho$ on the subsemiring $(E(S);\,+,\cdot)$, there is a congruence on $(S;\,+,\cdot)$ whose restriction to $E(S)$ is $\rho$.
\end{lemma}
\begin{proof}
Define a binary relation $\tau$ on $S$ by
\[
a\mathrel{\tau}b \Longleftrightarrow \exists~e\in E(S)~ea=eb\ \&\ e\mathrel{\rho}a^n\mathrel{\rho}b^n.
\]
We refer to the right-hand side by saying that $a\mathrel{\tau}b$ is \emph{witnessed} by $e\in E(S)$.

If $a,b\in E(S)$, then $a=a^n$ and $b=b^n$. Hence, $a\mathrel{\tau}b$ implies $a\mathrel{\rho}b$, that is, $\rho$ contains $\tau|_{E(S)}$. Conversely, if $a,b\in E(S)$ and $a\mathrel{\rho}b$, then $ab\in E(S)$, $aba=abb=ab$ since the idempotents $a$ and $b$ commute, and $ab\mathrel{\rho}a$ since $\rho$ is a congruence on $(E(S);\,+,\cdot)$. So, $a\mathrel{\tau}b$ is witnessed by the idempotent $ab$. Hence, $\rho$ is contained in $\tau|_{E(S)}$. Thus, $\rho=\tau|_{E(S)}$.

Clearly, the relation $\tau$ is symmetric. It is also reflexive as $a\mathrel{\tau}a$ is witnessed by the idempotent $a^n$. To check transitivity, suppose that $a\mathrel{\tau}b$ and $b\mathrel{\tau}c$ and let $e,f\in E(S)$ be the corresponding witnesses, that is,
\begin{gather*}
  ea=eb\ \&\ e\mathrel{\rho}a^n\mathrel{\rho}b^n,\\
  fb=fc\ \&\ f\mathrel{\rho}b^n\mathrel{\rho}c^n.
\end{gather*}
We have $ef\in E(S)$. Multiplying $ea=eb$ and $fb=fc$ on the left by $ef$ and using that the idempotents $e$ and $f$ commute, we obtain $efa=efc$. From $e\mathrel{\rho}b^n$ and $f\mathrel{\rho}b^n$, we conclude that $ef\mathrel{\rho}b^n$ since $\rho$ is a congruence on $(E(S);\,+,\cdot)$. By transitivity of $\rho$, we have $ef\mathrel{\rho}a^n\mathrel{\rho}c^n$. Hence, $a\mathrel{\tau}c$ is witnessed by the idempotent $ef$. Thus, $\tau$ is an equivalence.

It remains to verify that $\tau$ respects multiplication and addition. Suppose that $a\mathrel{\tau}b$ is witnessed by some $e\in E(S)$ and take any $c\in S$. Then $c^n$ is an idempotent and so is $c^ne$. Multiplying the equality $ea=eb$ on the left by $c^{n+1}$, we get $c^{n+1}ea=c^{n+1}eb$ whence
\[
c^neca\stackrel{\eqref{eq:center}}=c^{n+1}ea=c^{n+1}eb\stackrel{\eqref{eq:center}}=c^necb.
\]
From $e\mathrel{\rho}a^n\mathrel{\rho}b^n$, we conclude that $c^ne\mathrel{\rho}c^na^n\mathrel{\rho}c^nb^n$ since $\rho$ is a congruence on $(E(S);\,+,\cdot)$. We have
\[
c^na^n\stackrel{\eqref{eq:cryptic}}=(ca)^n\ \text{ and }\ c^nb^n\stackrel{\eqref{eq:cryptic}}=(cb)^n.
\]
Thus, $c^ne\mathrel{\rho}(ca)^n\mathrel{\rho}(cb)^n$, and we see that $ca\mathrel{\tau}cb$ is witnessed by the idempotent $c^ne$. Similarly, one verifies that $ac\mathrel{\tau}bc$. Thus, $\tau$ respects multiplication.

Checking that $\tau$ respects addition is based on expressing addition via multiplication and the operator $\mathrm{M}(\ )$; see \eqref{eq:summing}.

First, we prove the following implication:
\begin{equation}\label{002}
\text{If }\ f \in E(S)\ \text{ and }\  a\mathrel{\tau}f,\ \text{ then }\ a\mathrel{\tau}\mathrm{M}(a).
\end{equation}
Indeed, if $a\mathrel{\tau}f$ is witnessed by $e\in E(S)$, then $ea=ef$ yields $ea=e\cdot ef$, and $e\mathrel{\rho}a^n\mathrel{\rho}f$ ensures $e\mathrel{\rho}a^n\mathrel{\rho}ef$. Hence the idempotent $e$ also witnesses $a\mathrel{\tau}ef$. As we have already shown that $\tau$ respects multiplication, we have $\mathrm{M}(a)\stackrel{\eqref{21701}}=\mathrm{M}(a)a\mathrel{\tau}\mathrm{M}(a)ef$. From $ea=ef$, using that the idempotent $e$ commutes with $a$ by \eqref{eq:center}, we get
\[
\mathrm{M}(a)ef=(a+a^n)ef=eaf+(ea)^nf=ef+(ef)^nf=ef+ef=ef.
\]
Thus, we have $\mathrm{M}(a)\mathrel{\tau}ef$, which, combined with $a\mathrel{\tau}ef$, yields $a\mathrel{\tau}\mathrm{M}(a)$, as required.

Next, we show that $a\mathrel{\tau}b$ implies $\mathrm{M}(a)\mathrel{\tau}\mathrm{M}(b)$. Since $\tau$ respects multiplication, $a\mathrel{\tau}b$ implies $a\mathrm{M}(b)\mathrel{\tau}b\mathrm{M}(b)\stackrel{\eqref{21701}}=\mathrm{M}(b)$. Applying \eqref{002} with $a\mathrm{M}(b)$ in the role of $a$ and $\mathrm{M}(b)$ in the role of $f$ yields
\[
a\mathrm{M}(b)\mathrel{\tau}\mathrm{M}(a\mathrm{M}(b))=a\mathrm{M}(b)+(a\mathrm{M}(b))^n.
\]
We have
\begin{multline*}
a\mathrm{M}(b)+(a\mathrm{M}(b))^n\stackrel{\eqref{eq:cryptic}}=a\mathrm{M}(b)+a^n\mathrm{M}(b)^n\stackrel{\eqref{21701}}=a\mathrm{M}(b)+a^n\mathrm{M}(b)\\=(a+a^n)\mathrm{M}(b)=\mathrm{M}(a)\mathrm{M}(b).
\end{multline*}
It now follows that $a\mathrm{M}(b)\mathrel{\tau}\mathrm{M}(a)\mathrm{M}(b)$ whence $\mathrm{M}(b)\mathrel{\tau}\mathrm{M}(a)\mathrm{M}(b)$. Similarly, we obtain $\mathrm{M}(a)\mathrel{\tau}\mathrm{M}(b)\mathrm{M}(a)$. Since  $\mathrm{M}(a)$ and $\mathrm{M}(b)$ are idempotents, $\mathrm{M}(a)\mathrm{M}(b)=\mathrm{M}(b)\mathrm{M}(a)$. Thus, we have $\mathrm{M}(a)\mathrel{\tau}\mathrm{M}(b)$.

Finally, suppose that $a\mathrel{\tau}b$ and take any $c\in S$. Since $\tau$ respects multiplication, we have $ac^{n-1}\mathrel{\tau}bc^{n-1}$ whence, as we have just shown, $\mathrm{M}(ac^{n-1})\mathrel{\tau}\mathrm{M}(bc^{n-1})$.
This implies
\[
a+c\stackrel{\eqref{eq:summing}}= \mathrm{M}(ac^{n-1})c\mathrel{\tau}\mathrm{M}(bc^{n-1})c\stackrel{\eqref{eq:summing}}=b+c.
\]
Thus, $\tau$ respects addition.
\end{proof}

An algebra is called \emph{subdirectly irreducible} if it has a smallest non-diagonal congruence. Notice that our definition excludes the one-element algebra.

\begin{lemma}\label{lem:sisemifield}
In the variety $\mathbf{M}_n$, every subdirectly irreducible \ais{} is a semifield.
\end{lemma}

\begin{proof}
Let $(S;\,+,\cdot)$ be a subdirectly irreducible \ais{} in $\mathbf{M}_n$ and let $\theta$ denote its least non-diagonal congruence. The relation $\mu:=\theta|_{E(S)}$ is a congruence on the subsemiring $(E(S);\,+,\cdot)$. By Lemma~\ref{lem:restriction}, $\mu$ is not the diagonal of $E(S)$.

Let $\rho$ be any non-diagonal congruence on the subsemiring $(E(S);\,+,\cdot)$. By Lemma~\ref{lem:extension}, there exists a congruence $\tau$ on $(S;\,+,\cdot)$ such that $\rho=\tau|_{E(S)}$. The congruence $\tau$ is not the diagonal of $S$ since its restriction $\rho$ is not the diagonal of $E(S)$. Therefore, $\tau$ must contain $\theta$ as $\theta$ is the least non-diagonal congruence on $(S;\,+,\cdot)$. Then $\rho=\tau|_{E(S)}$ contains $\mu=\theta|_{E(S)}$. It follows that $\mu$ is the least non-diagonal congruence on $(E(S);\,+,\cdot)$, and this means that the \ais{} $(E(S);\,+,\cdot)$ is subdirectly irreducible.

As discussed at the beginning of this section, $(E(S);\,+,\cdot)$ is a duplicated semilattice. It is well known (and easy to verify) that, up to isomorphism, the only subdirectly irreducible semilattice is $(\{0,1\};\,\cdot)$ where $\cdot$ is the usual multiplication of 0 and~1. Its duplication is, up to isomorphism, the only subdirectly irreducible duplicated semilattice. This, in particular, proves the lemma for $n=1$. We may then assume that $n>1$, $E(S)=\{0,1\}$, and
\begin{equation}\label{eq:01}
0\cdot 0=0+0=0,\  0\cdot 1=0+1=0,\ 1\cdot 0=1+0=0,\ 1\cdot 1=1+1=1.
\end{equation}

Take an arbitrary element $s\in S$. By \eqref{eq:summing}, $s+1=\mathrm{M}(s\cdot 1)\cdot 1$.  On the other hand, applying \eqref{eq:summing} to $1+s$ yields $1+s=\mathrm{M}(1\cdot s^{n-1})s$. Since $s+1=1+s$, we have
\begin{equation}\label{eq:m}
\mathrm{M}(s\cdot 1)\cdot 1=\mathrm{M}(1\cdot s^{n-1})s.
\end{equation}
By Lemma~\ref{lem:m}, both $\mathrm{M}(s\cdot 1)$ and $\mathrm{M}(1\cdot s^{n-1})$ are idempotents, and thus belong to $\{0,1\}$. Hence \eqref{eq:01} yields
\[
0\cdot\mathrm{M}(s\cdot 1)=0\cdot\mathrm{M}(1\cdot s^{n-1})=0.
\]
Therefore, multiplying \eqref{eq:m} on the left by 0 yields $0\cdot 1=0\cdot s$. In view of \eqref{eq:01} and \eqref{eq:center}, this shows that 0 is a multiplicative zero in $(S;\,+,\cdot)$.

The semigroup $(S;\,\cdot)$ is a union of subgroups whose identity elements are the idempotents of $(S;\,\cdot)$. Since $E(S)=\{0,1\}$, the semigroup has exactly two maximal subgroups: one with identity element 0 and one with identity element 1. As 0 has been shown to be a multiplicative zero, the former subgroup reduces to $\{0\}$ while the latter consists of all nonzero elements. Hence, $(S;\,+,\cdot)$ is a semifield.
\end{proof}

We now describe the subdirectly irreducible \ais{}s in the variety $\mathbf{M}_n$.

\begin{proposition}\label{prop:si}
For any \ais{} $(S;\,+,\cdot)$ in the variety $\mathbf{M}_n$, the following are equivalent:
\begin{itemize}
  \item [(i)]   $(S;\,+,\cdot)$ is subdirectly irreducible;
  \item [(ii)]  $(S;\,+,\cdot)$ is a semifield;
  \item [(iii)] $(S;\,+,\cdot)$ is the flat extension of a group in $\mathbf{G}_n$;
  \item [(iv)]  $(S;\,+,\cdot)$ is simple and $|S|>1$.
\end{itemize}
\end{proposition}

\begin{proof} (i) $\Rightarrow$ (ii) is Lemma~\ref{lem:sisemifield}.

(ii) $\Rightarrow$ (iii) follows from Corollary~\ref{cor:semifield}.

(iii) $\Rightarrow$ (iv) is Lemma~\ref{lem:flatsimple}.

(iv) $\Rightarrow$ (i) is obvious.
\end{proof}

Proposition~\ref{prop:si} and Birkhoff's Subdirect Decomposition Theorem \cite[Theorem II.8.6]{BuSa81} lead to the following structural description of the \ais{}s in $\mathbf{M}_n$.

\begin{corollary}
\label{cor:structure}
Every \ais{} with more than one element in the variety $\mathbf{M}_n$ is isomorphic to a subdirect product of flat extensions of groups from the variety $\mathbf{G}_n$.
\end{corollary}

\section{{Embedding $L_q(\mathbf{G}_n)$ into $L(\mathbf{M}_n)$}{}}
\label{sec:embedding}

For a variety $\mathbf{V}$ of \ais{}s, let $\mathbf{V}^\sharp$ stand for the class of all groups whose flat extensions lie in $\mathbf{V}$.

\begin{lemma}
\label{lem:v2g}
$\mathbf{V}^\sharp$ is a quasivariety of groups for each \ais{} variety $\mathbf{V}$ containing $\mathbf{M}_1$.
\end{lemma}

\begin{proof}
We have to verify that $\mathbf{V}^\sharp$ contains the one-element group $E$ and is closed under forming isomorphic copies of subgroups, direct products and ultraproducts.\footnote{With possible further applications in mind, we treat groups as $\{\cdot,{}^{-1}\}$-algebras in this lemma, even though we only apply the operator $\mathbf{V}\mapsto\mathbf{V}^\sharp$ below in the case where all groups involved have exponent dividing $n$, so that group inversion is expressible via multiplication.}

Since the flat extension $E^\flat$ of $E$ is isomorphic to the duplicated semilattice  $\langle\{0,1\};\,+,\cdot\rangle$ with the operations defined by \eqref{eq:01}, it lies in the variety $\mathbf{M}_1$. Hence, $E\in \mathbf{V}^\sharp$ for every \ais{} variety $\mathbf{V}$ containing $\mathbf{M}_1$.

Let $H$ be a group and $\varphi$ an isomorphic embedding of $H$ into a group $G\in\mathbf{V}^\sharp$. Define $\varphi^\flat\colon H^\flat\to G^\flat$ by setting, for every $a\in H\cup\{0\}$,
\[
a\varphi^\flat:=
\begin{cases}
  a\varphi&\text{if } a\in H,\\
  0&\text{if } a=0.
\end{cases}
\]
It is easy to check that $\varphi^\flat$ preserves addition and multiplication, so $H^\flat$ is isomorphic to a subsemiring in $G^\flat$. The \ais{} $G^\flat$ belongs to the variety $\mathbf{V}$ by the definition of the class $\mathbf{V}^\sharp$. Varieties are closed under taking subsemirings, whence $H^\flat$ also lies in $\mathbf{V}$. Therefore, the group $H$ belongs to $\mathbf{V}^\sharp$.

To establish closure under direct products, let $\{G_i\}_{i \in I}$ be any nonempty family of groups with each $G_i\in\mathbf{V}^\sharp$ and let $G:=\prod_{i \in I}G_i$ be the direct product of the family. By the definition of the class $\mathbf{V}^\sharp$, all \ais{}s $G_i^\flat$ belong to the variety $\mathbf{V}$. Since varieties are closed under direct products, the direct product $\prod_{i \in I}G^\flat_i$ also lies in $\mathbf{V}$. Define $\psi\colon\prod_{i \in I}G^\flat_i\to G^\flat$ by setting, for every $I$-tuple $(a_i)$,
\[
(a_i)\psi:=
\begin{cases}
  (a_i)&\text{if } a_i\in G_i \text{ for all } i\in I,\\
  0&\text{if } a_i \text{ is the zero of $G^\flat_i$ for some } i\in I.
\end{cases}
\]
It is easy to verify that $\psi$ is a semiring homomorphism onto $G^\flat$. Varieties are closed under homomorphic images, so $G^\flat$ belongs to $\mathbf{V}$. Hence, the group $G$ belongs to $\mathbf{V}^\sharp$.

It remains to show that $\mathbf{V}^\sharp$ is closed under ultraproducts. Recall the definition of an ultraproduct.

A \emph{filter} $\mathcal{F}$ on a set $I$ is a nonempty set of nonempty subsets of $I$ such that:
\begin{itemize}
  \item if $K\in\mathcal{F}$ and $K\subseteq L\subseteq I$, then $L\in\mathcal{F}$;
  \item if $K,L\in\mathcal{F}$, then $K\cap L\in\mathcal{F}$.
\end{itemize}
A filter $\mathcal{F}$ is called an \emph{ultrafilter} if
\begin{itemize}
  \item for every $K\subseteq I$, exactly one of the sets $K$ and $I\setminus K$ belongs to $\mathcal{F}$.
\end{itemize}

Let $\{A_i\}_{i \in I}$ be a family of algebras of the same signature and $\mathcal{F}$ a filter on the index set $I$. The \emph{reduced product} of $\{A_i\}_{i \in I}$ over $\mathcal{F}$ is the quotient of the direct product $\prod_{i \in I}A_i$ by the congruence $\sim_\mathcal{F}$ defined by
\[
(a_i)\sim_\mathcal{F}(b_i) \ \text{ if and only if }\ \{i\in I \mid a_i=b_i\}\in \mathcal{F}.
\]
An \emph{ultraproduct} is the reduced product over an ultrafilter.

Now take any family $\{G_i\}_{i \in I}$ of groups with each $G_i\in\mathbf{V}^\sharp$ and any ultrafilter $\mathcal{U}$ on $I$. Form the ultraproduct $G$ of $\{G_i\}_{i \in I}$ over $\mathcal{U}$. Let $(S;\,+,\cdot)$ stand for the ultraproduct of the family $\{G^\flat_i\}_{i \in I}$ of the flat extensions of the groups $G_i$, taken over the same ultrafilter $\mathcal{U}$.

By the definition of the class $\mathbf{V}^\sharp$, all \ais{}s $G_i^\flat$ belong to the variety $\mathbf{V}$. Since varieties are closed under direct products and homomorphic images, the \ais{} $(S;\,+,\cdot)$, being a quotient of the direct product $\prod_{i \in I}G^\flat_i$, also lies in $\mathbf{V}$.

Consider the map $\iota$ that sends every $I$-tuple from the direct product $\prod_{i \in I}G_i$ to the same $I$-tuple, now viewed as an element of the direct product $\prod_{i \in I}G^\flat_i$. Clearly,
\[
(a_i)\sim_\mathcal{U}(b_i)\ \text{ if and only if }\ (a_i)\iota\sim_\mathcal{U}(b_i)\iota.
\]
This ensures that the map $G\to S$, which sends the $\sim_\mathcal{U}$-class of $(a_i)$ to the $\sim_\mathcal{U}$-class of $(a_i)\iota$, is well-defined and injective. It is clear that the map preserves multiplication, whence its image $H$ is a subgroup of the multiplicative reduct $(S;\cdot)$ isomorphic to the group $G$.

Let $\mathbf{0}$ be the $\sim_\mathcal{U}$-class of the $I$-tuple $(0_i)$ whose $i$-th entry is the zero $0_i$ of $G^\flat_i$ for all $i\in I$. It is easy to see that the set $H\cup\{\mathbf{0}\}\subseteq S$ is closed under multiplication. We claim that it is also closed under addition, for which it suffices to show that the sum of any two distinct elements of $H$ is $\mathbf{0}$. That is, we must verify that the $\sim_\mathcal{U}$-class of $(a_i)\iota+(b_i)\iota$ is $\mathbf{0}$ whenever $(a_i)\nsim_\mathcal{U}(b_i)$. If $(a_i)\nsim_\mathcal{U}(b_i)$, then by the definition of the congruence $\sim_\mathcal{U}$, the set $K:=\{i\in I \mid a_i=b_i\}$ does not belong to $\mathcal{U}$. Since $\mathcal{U}$ is an ultrafilter, it follows that the set $I\setminus K=\{i\in I \mid a_i\ne b_i\}$ is in $\mathcal{U}$. For each $i\in I\setminus K$, we have $a_i+b_i=0_i$ in $G^\flat_i$, so $(a_i)\iota+(b_i)\iota\sim_\mathcal{U}(0_i)$.

Thus, $(H\cup\{\mathbf{0}\};\,+,\cdot)$ is a subsemiring of the \ais{} $(S;\,+,\cdot)$, and it is easy to see that this subsemiring is isomorphic to the flat extension $G^\flat$ of $G$. Since varieties are closed under taking subsemirings,  it follows that $G^\flat \in \mathbf{V}$. Thus, the ultraproduct $G$ lies in $\mathbf{V}^\sharp$.
\end{proof}

Given a semigroup quasivariety $\mathbf{Q}$ consisting of groups, denote by
\[
\mathbf{Q}^\flat:=\var\{G^\flat\mid G\in\mathbf{Q}\}
\]
the variety of \ais{}s generated by the flat extensions of groups from $\mathbf{Q}$.

\begin{theorem}
\label{thm:main}
Let $n$ be a positive integer. The maps $\flat\colon\mathbf{Q}\mapsto\mathbf{Q}^\flat$ and $\sharp\colon\mathbf{V}\mapsto\mathbf{V}^\sharp$ are mutually inverse complete isomorphisms between the lattice $L_q(\mathbf{G}_n)$ of all quasivarieties of groups of exponent dividing $n$ and the interval of the lattice of \ais{} varieties consisting of all varieties that contain the variety $\mathbf{M}_1$ and are contained in the variety $\mathbf{M}_n$. These maps preserve finite axiomatizability and nonfinite axiomatizability.
\end{theorem}

\begin{proof}
For any group $G$ in the variety $\mathbf{G}_n$, its flat extension $G^\flat$ clearly satisfies identity \eqref{eq:srn}. Therefore, the $n$-th powers are precisely the idempotents in $G^\flat$. By construction, the flat extension of any group has exactly two idempotents: the identity element of the group and 0. Hence, when adding or multiplying two $n$-th powers, the result is always 0, unless both operands are equal to the identity element of the group. In either case, the sum of two $n$-th powers equals their product. This shows that $G^\flat$ satisfies identity \eqref{new3}.

We have thus proved that the flat extension of every group from $\mathbf{G}_n$ belongs to the variety $\mathbf{M}_n$. Hence, $\mathbf{Q}^\flat$ is contained in $\mathbf{M}_n$ for every quasivariety $\mathbf{Q}\in L_q(\mathbf{G}_n)$.

The variety $\mathbf{M}_1$ is known to be a minimal nontrivial semiring variety; see \cite{Polin80} where this variety appears as $\mathfrak{J}$.  Consequently, $\mathbf{M}_1$ is generated by any of its nontrivial members. In particular, the flat extension $E^\flat$ of the one-element group $E$ generates $\mathbf{M}_1$. Therefore, $\mathbf{Q}^\flat$ contains $\mathbf{M}_1$ for every $\mathbf{Q}\in L_q(\mathbf{G}_n)$.

Denote the interval of the lattice of \ais{} varieties consisting of all varieties that contain $\mathbf{M}_1$ and are contained in $\mathbf{M}_n$ by $[\mathbf{M}_1,\mathbf{M}_n]$. We see that this interval includes the image of the lattice $L_q(\mathbf{G}_n)$ under the map $\flat:\mathbf{Q}\mapsto\mathbf{Q}^\flat$.

Take any variety $\mathbf{V}\in[\mathbf{M}_1,\mathbf{M}_n]$. Every nontrivial variety is generated by its subdirectly irreducible members; see \cite[Corollary II.9.7]{BuSa81}. This applies to $\mathbf{V}$ as well, but by Proposition~\ref{prop:si}, the subdirectly irreducible members are flat extensions of groups. According to Lemma~\ref{lem:v2g}, groups whose flat extensions lie in $\mathbf{V}$ form a quasivariety, denoted by $\mathbf{V}^\sharp$. Hence, $\mathbf{V}$ is generated by the flat extensions of groups from $\mathbf{V}^\sharp$, that is, $\mathbf{V}=(\mathbf{V}^\sharp)^\flat$.

Groups satisfy the semigroup identities of their flat extensions, and therefore, groups in $\mathbf{V}^\sharp$ satisfy identity \eqref{eq:srn}. The exponent of any group satisfying \eqref{eq:srn} divides $n$, so the quasivariety $\mathbf{V}^\sharp$ is contained in $\mathbf{G}_n$. We have thus proved that $\flat:\mathbf{Q}\mapsto\mathbf{Q}^\flat$ maps the lattice $L_q(\mathbf{G}_n)$ onto the interval $[\mathbf{M}_1,\mathbf{M}_n]$.

Take any quasivariety $\mathbf{Q}\in L_q(\mathbf{G}_n)$. We aim to show that $(\mathbf{Q}^\flat)^\sharp=\mathbf{Q}$. The quasivariety $(\mathbf{Q}^\flat)^\sharp$ consists of all groups whose flat extensions belong to the variety $\mathbf{Q}^\flat$ generated by the flat extensions of groups in $\mathbf{Q}$. Hence, every group from $\mathbf{Q}$ lies in $(\mathbf{Q}^\flat)^\sharp$. It remains to prove that every group $G$ from $(\mathbf{Q}^\flat)^\sharp$ lies in $\mathbf{Q}$. This is clear when $G$ is the one-element group, so we assume that $G$ has more than one element. It is known (and easy to verify) that a group belongs to a quasivariety if all of its finitely generated subgroups do. Hence we may assume that $G$ is finitely generated.

The flat extension $G^\flat$ belongs to the variety $\mathbf{Q}^\flat$. By the HSP Theorem, $G^\flat$ is a homomorphic image of an \ais{} $(S;\,+,\cdot)$ that embeds in a direct product of a family $\{G^\flat_i\}_{i\in I}$ of flat extensions of groups $G_i\in\mathbf{Q}$. For every idempotent $e\in E(S)$, let $H_e=\{h\in S\mid h^n=e\}$. Then $H_e$ is a group whose identity element is $e$, and the Clifford semigroup $(S;\,\cdot)$ is a semilattice of  groups $H_e$. Denote the projection $\prod_{i \in I}G^\flat_i\to G^\flat_i$ by $\pi_i$. Restricting these projections to the subgroup $H_e$ of $(S;\,\cdot)$, we obtain a decomposition of $H_e$ as a subdirect product of the groups $H_e\pi_i$ each of which is either a subgroup of the group $G_i$ or the one-element group. Since quasivarieties are closed under subgroups and direct products, all groups $H_e$, $e\in E(S)$, belong to the quasivariety $\mathbf{Q}$.

Denote the homomorphism from $(S;\,+,\cdot)$ onto $G^\flat$ by $\zeta$. Since the group $G$ is finitely generated, we can fix some generators $r_1,\dots,r_k$ of $G$ and choose a preimage $q_i$ under $\zeta$ for each $r_i$. By Lemma~\ref{lem:e(s)}, $q_i^n$ is an idempotent for each $i=1,\dots,k$. Let $e_0:=q_1^n\cdots q_k^n$; then $e_0$ is an idempotent (by Lemma~\ref{lem:e(s)}) and $e_0\zeta\in G$, whence $e_0\zeta$ is the identity element $e_G$ of the group $G$. Since the exponent of $G$ is finite, every element $g\in G$ can be represented as a product of the generators $r_1,\dots,r_k$, that is, $g=w(r_1,\dots,r_k)$, where $w(x_1,\dots,x_k)$ is a word built from variables that belong to $\{x_1,\dots,x_k\}$. Then the element $h:=w(e_0q_1,\dots,e_0q_k)$ is such that $h\zeta=g$. Indeed, the idempotent $e_0$ commutes with every element of $S$ by \eqref{eq:center}, and therefore,
\[
h\zeta=(w(e_0q_1,\dots,e_0q_k))\zeta\stackrel{\eqref{eq:center}}=(w(q_1,\dots,q_k)e_0)\zeta=w(r_1,\dots,r_k)\cdot e_0\zeta=ge_G=g,
\]
since $w(r_1,\dots,r_k)=g$ by the choice of the word $w(x_1,\dots,x_k)$. Using that $q_i^n$ is an idempotent, we have
\[
e_0q_i^n=q_1^n\cdots q_i^n\cdots q_k^n\cdot q_i^n\stackrel{\eqref{eq:center}}=q_1^n\cdots (q_i^n)^2\cdots q_k^n=q_1^n\cdots q_i^n\cdots q_k^n=e_0.
\]
Hence
\[
h^n=w(e_0q_1,\dots,e_0q_k)^n\stackrel{\eqref{eq:cryptic}}=w(e_0q_1^n,\dots,e_0q_k^n)=w(e_0,\dots,e_0)=e_0,
\]
which means that $h\in H_{e_0}$. Thus, every element $g\in G$ has a preimage under $\zeta$ in the group $H_{e_0}$. It is clear that $H_{e_0}\zeta\subseteq G$, so $H_{e_0}\zeta=G$.

Let $F:=\{f\in E(S)\mid f\zeta=e_G\ \text{ and }\ e_0f=f\}$. The set $F$ is nonempty since it contains the idempotent $e_0$. We equip $F$ with a binary relation $\leqslant$ by setting
\[
e\leqslant f \ \text{ if and only if }\ ef=e.
\]
It is known (and easy to verify) that  $\leqslant$ is a partial order. For $f\in F$, denote by $f{\downarrow}$ the lower set of $f$, that is, the set $\{e\in F\mid e\leqslant f\}$. Let $\mathcal{F}$ consist of all nonempty subsets $L$ of $F$ such that $f{\downarrow}\subseteq L$ for some $f\in F$. It is easy to see that $\mathcal{F}$ is a filter on $F$---closure under supersets is clear, and closure under intersection follows from the equality $(f_1{\downarrow})\cap(f_2{\downarrow})=(f_1f_2){\downarrow}$, as it is easy to see that $f_1f_2\in F$ for all $f_1,f_2\in F$.

Form the reduced product $H$ of $\{H_f\}_{f\in F}$ over the filter $\mathcal{F}$. By construction, $H$ is the quotient group of the direct product $D:=\prod_{f \in F}H_f$ by the congruence
\[
(h_f)_{f\in F}\sim_\mathcal{F}(h'_f)_{f\in F} \ \text{ if and only if }\ \{f\in F \mid h_f=h'_f\}\in\mathcal{F}.
\]
We let $N$ denote the kernel of the natural homomorphism $D\twoheadrightarrow H$. Since all groups $H_f$, $f\in F$, belong to the quasivariety $\mathbf{Q}$, and quasivarieties are closed under reduced products (see \cite[Theorem V.2.7]{BuSa81}), we conclude that the group $H$ belongs to $\mathbf{Q}$.

For all $h\in H_{e_0}$ and $f\in F$, we have $hf\in H_f$ since $(hf)^n=h^nf=e_0f=f$. Define $\vec{h}\in D$ by setting the $f$-th coordinate of $\vec{h}$ to $hf$. The map $h\mapsto \vec{h}$ is injective as $h=he_0$ is the $e_0$-th coordinate of $\vec{h}$. The set $\vec{H}:=\{\vec h\mid h\in H_{e_0}\}$ is easily seen to be a subgroup of the direct product $D$. The map $\vec{\zeta}$ defined by $\vec{h}\vec{\zeta}:=h\zeta$ is a homomorphism from this subgroup onto the group $G$ since $H_{e_0}\zeta=G$, as shown above. We are going to prove that the kernel $K$ of $\vec{\zeta}$ is equal to $N\cap\vec{H}$.

First, take an arbitrary element $\vec{h}\in N\cap\vec{H}$ and show that $\vec{h}\in K$, that is, $\vec{h}\vec{\zeta}=e_G$. The subgroup $N$ consists of all elements $(h_f)_{f\in F}$ of the direct product $D$ that are $\sim_\mathcal{F}$-related to the identity element $(f)_{f\in F}$ of $D$. By definition of the congruence $\sim_\mathcal{F}$, the set $\{f\in F\mid h_f=f\}$ belongs to the filter $\mathcal{F}$. In particular, this set is nonempty. Thus, let $f_0\in F$ be such that the $f_0$-th coordinate of $\vec{h}$ is equal to $f_0$. Since, by definition, the $f_0$-th coordinate of $\vec{h}$ is $hf_0$, we conclude that $hf_0=f_0$. Now we compute
\begin{align*}
  \vec{h}\vec{\zeta}&=h\zeta &&\text{by definition of $\vec{\zeta}$}\\
                    &=h\zeta f_0\zeta&&\text{since $f_0\zeta=e_G$ by definition of $F$}\\
                    &=(hf_0)\zeta&&\text{since $\zeta$ is a homomorphism}\\
                    &=f_0\zeta&&\text{since $hf_0=f_0$}\\
                    &=e_G,&&
\end{align*}
as required.

Conversely, take an arbitrary element $\vec{h}\in K$. Then $\vec{h}\vec{\zeta}=e_G$, whence $h\zeta=e_G$ by definition of $\vec{\zeta}$. Since $e_0\zeta=e_G$ and $\zeta$ is a semiring homomorphism, we conclude that
$(h+e_0)\zeta=h\zeta+e_0\zeta=e_G+e_G=e_G$. For $h\in H_{e_0}$, we have $h^n=e_0$, so $h+e_0=h+h^n=\mathrm{M}(h)$. By Lemma~\ref{lem:m}, $\mathrm{M}(h)$ is an idempotent and $h\mathrm{M}(h)=\mathrm{M}(h)$.  Hence, $e_0\mathrm{M}(h)=h^n\mathrm{M}(h)=\mathrm{M}(h)$, and since $\mathrm{M}(h)\zeta=e_G$, we have $\mathrm{M}(h)\in F$. If $f\in\mathrm{M}(h){\downarrow}$, then $f\mathrm{M}(h)=\mathrm{M}(h)f=f$, and hence $hf=h\mathrm{M}(h)f=\mathrm{M}(h)f=f$. We see that the set $\{f\in F\mid hf=f\}$ contains the lower set $\mathrm{M}(h){\downarrow}$ and thus belongs to the filter $\mathcal F$. Since, by definition, $hf$ is the $f$-th coordinate of $\vec{h}$, this proves that $\vec{h}$ is $\sim_\mathcal{F}$-related to the identity element $(f)_{f\in F}$ of $D$, and so $\vec{h}\in N\cap\vec{H}$.

Once the equality $K=N\cap\vec{H}$ is proved, the standard isomorphism theorems of group theory yield
\[
G\cong\vec{H}/K=\vec{H}/(N\cap\vec{H})\cong\vec{H}N/N,
\]
so $G$ is isomorphic to a subgroup of $D/N\cong H$. Since $H$ belongs to $\mathbf{Q}$ and quasivarieties are closed under subgroups, we see that $G$ indeed lies in $\mathbf{Q}$. The equality $(\mathbf{Q}^\flat)^\sharp=\mathbf{Q}$ is thus established.

From the equalities $(\mathbf{Q}^\flat)^\sharp=\mathbf{Q}$ and $\mathbf{V}=(\mathbf{V}^\sharp)^\flat$, we conclude that the maps $\flat\colon\mathbf{Q}\mapsto\mathbf{Q}^\flat$ and $\sharp\colon\mathbf{V}\mapsto\mathbf{V}^\sharp$ are mutually inverse bijections between the lattice $L_q(\mathbf{G}_n)$ and the interval $[\mathbf{M}_1,\mathbf{M}_n]$. Each of these maps clearly preserves class inclusions, and it is well known that every order-preserving bijection with order-preserving inverse between two complete lattices is a complete lattice isomorphism.

That the maps $\flat$ and $\sharp$ preserve finite axiomatizability and nonfinite axiomatizability is an immediate consequence. Recall that an element $c$ of a complete lattice $(L;\,\vee,\wedge)$ is said to be \emph{cocompact} if
$$
c \ge \bigwedge_{j \in J} x_j\,,\ \ x_j \in L,
$$
implies
$$
c\ge x_{j_1} \wedge \cdots \wedge x_{j_k}
$$
for some finite subset $\{j_1,\ldots,j_k\} \subset J$. Since $\flat$ and $\sharp$ are complete lattice isomorphisms between the lattice $L_q(\mathbf{G}_n)$ and the interval $[\mathbf{M}_1,\mathbf{M}_n]$, they preserve cocompactness. It is well-known (and easy to verify) that a sub(quasi)variety $\mathbf{C}$ of a (quasi)variety $\mathbf{D}$ is finitely axiomatizable within $\mathbf{D}$ if and only if $\mathbf{C}$ is a cocompact element in the lattice of all sub(quasi)varieties of $\mathbf{D}$. Hence a quasivariety $\mathbf{Q}\subseteq\mathbf{G}_n$ is finitely axiomatizable within $\mathbf{G}_n$ if and only if the variety $\mathbf{Q}^\flat$ is finitely axiomatizable  within $\mathbf{M}_n$.  Finally, since both $\mathbf{G}_n$ and $\mathbf{M}_n$ are finitely axiomatizable, any subclass that is finitely axiomatizable within $\mathbf{G}_n$ or $\mathbf{M}_n$ is finitely axiomatizable.
\end{proof}

\section{Applications}

\subsection{{The number of subvarieties of the variety $\mathbf{Sr}_n$}{}}

The subvarieties of the varieties $\mathbf{Sr}_1$ and $\mathbf{Sr}_2$ were classified in \cite{GPZ05,Pas05} and \cite{RZW17}, respectively. In \cite{GPZ05,Pas05}, it was shown that $\mathbf{Sr}_1$---that is, the variety of all multiplicatively idempotent \ais{}s---has exactly 78 subvarieties. In \cite{RZW17}, the classification was extended to the subvarieties of $\mathbf{Sr}_2$, whose number was established to be 179. The natural question of whether a further extension to $\mathbf{Sr}_3$, $\mathbf{Sr}_4$, etc., is possible was one of the initial motivations for the present paper. Our first application demonstrates a drastic change in the number of subvarieties of $\mathbf{Sr}_n$ that occurs when $n$ rises from 2 to 3.

\begin{corollary}
\label{cor:cardinal}
For every $n\ge 3$, the variety $\mathbf{M}_n$ \textup(and hence, the variety $\mathbf{Sr}_n$\textup) has continuum many subvarieties.
\end{corollary}

\begin{proof}
By Theorem~\ref{thm:main}, it suffices to prove that for every $n\ge 3$, the variety $\mathbf{G}_n$ of all groups of exponent dividing $n$ has continuum many subquasivarieties. The latter fact follows from a characterization of locally finite semigroup varieties with countably many subquasivarieties obtained by Mark Sapir~\cite{Sapir84}, more precisely, for its specialization to locally finite group varieties. Recall that a variety is \emph{locally finite} if each of its finitely generated members is finite. Part of \cite[Theorem 1]{Sapir84} asserts that if a locally finite group variety has countably many subquasivarieties, then all nilpotent groups in the variety must be abelian.

Every variety generated by a finite group is known to be locally finite; see \cite[Theorem 10.16]{BuSa81}. Therefore, Sapir's result quoted in the previous paragraph ensures that for any finite nilpotent non-abelian group $N$, the variety $\var\{N\}$ has continuum many subquasivarieties. It remains to verify that for every $n\ge 3$, the variety $\mathbf{G}_n$ contains a finite nilpotent non-abelian group. Any number $n\ge 3$ is divisible by either an odd prime $p$ or 4. In the former case, $\mathbf{G}_n$ contains the \emph{Heisenberg group modulo $p$}, that is, the $p^3$-element group given by the presentation
\begin{equation}\label{eq:heisenberg}
H_p:=\langle a,b,c \mid a^p=b^p=c^p=1,\, ac=ca,\ bc=cb,\ ab=bac \rangle.
\end{equation}
It is known (and easy to see) that for each odd prime $p$, the group $H_p$ is nilpotent and non-abelian. If $n$ is divisible by 4, then $\mathbf{G}_n$ contains the 8-element \emph{quaternion group}
\[
Q_8:=\langle a,b \mid a^4 = 1,\ a^2 = b^2,\ b=aba\rangle,
\]
which also is nilpotent and non-abelian.
\end{proof}

The classification of the subvarieties of the varieties $\mathbf{Sr}_1$ and $\mathbf{Sr}_2$ obtained in \cite{GPZ05,Pas05,RZW17} shows that each subvariety of $\mathbf{Sr}_1$ or $\mathbf{Sr}_2$ is generated by a finite \ais{} and is axiomatized by finitely many identities. Since both the set of all isomorphism types of finite \ais{}s and the set of all finite collections of semiring identities are countable, Corollary~\ref{cor:cardinal} readily implies that none of these properties extend to the subvarieties of $\mathbf{Sr}_n$ for $n\ge 3$.

\subsection{{Non-modularity of the lattice $L(\mathbf{Sr}_n)$ for certain $n$}{}}

The lattices $L(\mathbf{Sr}_1)$ and $L(\mathbf{Sr}_2)$ were shown to be distributive in \cite{GPZ05,Pas05} and \cite{RZW17}, respectively. It is natural to ask whether distributivity or at least modularity persists in the lattices $L(\mathbf{Sr}_n)$ for $n\ge 3$. For context, we mention that the lattice of all varieties of semigroups satisfying identity \eqref{eq:srn}, which defines the variety $\mathbf{Sr}_n$, is modular for every $n$; see \cite{Pas90,Pas91} and~\cite{PR90}.

The lattice isomorphism established in Theorem~\ref{thm:main} suggests exploring the lattices $L_q(\mathbf{G}_n)$. Unfortunately, current knowledge about lattices of group quasivarieties is insufficient to fully resolve the question we stated. What we have at the moment is the following partial result.

\begin{corollary}
\label{cor:nonmodular}
For every $n$ divisible by either the square of an odd prime or by $4p$, where $p$ is any prime, the variety $\mathbf{M}_n$ \textup(and hence, the variety $\mathbf{Sr}_n$\textup) has non-modular subvariety lattice.
\end{corollary}

\begin{proof}
Let $p$ be an odd prime. An example of a non-modular sublattice in the lattice $L_q(\mathbf{G}_{p^2})$ is exhibited in \cite[Example 1]{BuGo75}; see \cite[Theorem 3.1.15]{Budkin:2002} for a detailed proof. Since both the paper \cite{BuGo75} and the monograph \cite{Budkin:2002} were published in Russian and may not be readily accessible to the reader, we reproduce the example here. Our notation differs from that in \cite[Example 1]{BuGo75} and \cite[Theorem 3.1.15]{Budkin:2002}.

Let $C_p$ and $C_{p^2}:=\langle c\mid c^{p^2}=1\rangle$ be the cyclic groups with $p$ and $p^2$ elements, respectively. Consider the direct product $H_p\times C_{p^2}$, where $H_p$ is the Heisenberg group modulo $p$ defined by \eqref{eq:heisenberg}. Clearly, the exponent of this product is $p^2$, whence it lies in the variety $\mathbf{G}_{p^2}$. Let $\widetilde{H}_p$ stand for the subgroup of $H_p\times C_{p^2}$, generated by the pairs $(a,c)$ and $(b,1)$. Then the quasivarieties generated by each of the groups $C_p$, $C_{p^2}$, $H_p$, $\widetilde{H}_p$, and $H_p\times C_{p^2}$, form a 5-element non-modular sublattice in $L_q(\mathbf{G}_{p^2})$; see Figure~\ref{fig:pentagon}(left).

\begin{figure}[ht]
\begin{center}
\unitlength=1mm
\begin{picture}(35,25)(0,0)
\put(10,10){\line(2,-1){10}}
\put(20,5){\line(1,1){10}}
\put(20,25){\line(1,-1){10}}
\put(10,10){\line(0,1){10}}
\put(10,20){\line(2,1){10}}
\put(10,10){\circle*{1}}
\put(20,25){\circle*{1}}
\put(10,20){\circle*{1}}
\put(20,5){\circle*{1}}
\put(30,15){\circle*{1}}
\put(21,3){$\qvar\{C_p\}$}
\put(-5,8){$\qvar\{C_{p^2}\}$}
\put(-4.7,20.5){$\qvar\{\widetilde{H}_p\}$}
\put(31,13){$\qvar\{H_p\}$}
\put(21,26){$\qvar\{H_p\times C_{p^2}\}$}
\end{picture}
\qquad\qquad\qquad\qquad
\begin{picture}(35,25)(0,0)
\put(10,10){\line(2,-1){10}}
\put(20,5){\line(1,1){10}}
\put(20,25){\line(1,-1){10}}
\put(10,10){\line(0,1){10}}
\put(10,20){\line(2,1){10}}
\put(10,10){\circle*{1}}
\put(20,25){\circle*{1}}
\put(10,20){\circle*{1}}
\put(20,5){\circle*{1}}
\put(30,15){\circle*{1}}
\put(21,3){$\qvar\{C_{2p}\}$}
\put(-5,8){$\qvar\{C_{4p}\}$}
\put(-6.2,20.5){$\qvar\{\widetilde{D}_{2p}\}$}
\put(31,13){$\qvar\{D_{2p}\}$}
\put(21,26){$\qvar\{D_{2p}\times C_{4p}\}$}
\end{picture}
\caption{5-element non-modular sublattices in the lattices $L_q(\mathbf{G}_{p^2})$ (left) and $L_q(\mathbf{G}_{4p})$ (right), where $p$ is an odd prime}\label{fig:pentagon}
\end{center}
\end{figure}
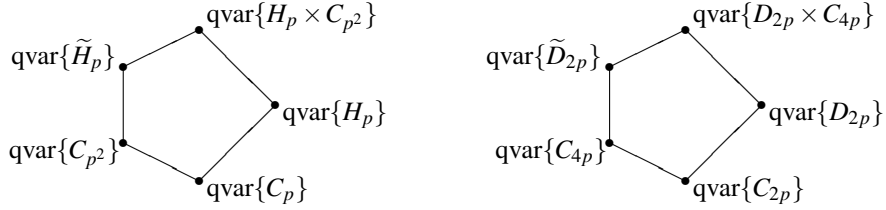

An example of a non-modular sublattice in the lattice $L_q(\mathbf{G}_{4p})$ can be extracted from \cite[Example 2]{BuGo75}. We present it here, using much simplified notation.

Consider the \emph{dihedral group} of order $2p$ given by the presentation
\[
D_{2p}:=\langle a,b \mid a^p=b^2=1,\ ab=ba^{p-1}\rangle.
\]
If $C_{4p}:=\langle c\mid c^{4p}=1\rangle$ is the $4p$-element cyclic group, then the direct product $D_{2p}\times C_{4p}$ has exponent $4p$, and so it lies in the variety $\mathbf{G}_{4p}$.  Let $\widetilde{D}_{2p}$ stand  for the subgroup of $D_{2p}\times C_{4p}$ generated by the pairs $(a,1)$ and $(b,c^p)$. Then the quasivarieties generated by each of the following groups: the $2p$-element cyclic group $C_{2p}$, $C_{4p}$, $D_{2p}$, $\widetilde{D}_{2p}$, and $D_{2p}\times C_{4p}$, form a 5-element non-modular sublattice in $L_q(\mathbf{G}_{4p})$; see Figure~\ref{fig:pentagon}(right).

It remains to present a non-modular sublattice in $L_q(\mathbf{G}_{8})$. Such an example (constructed by F\"edorov~\cite{Fed}) is reproduced in the monograph \cite{Budkin:2002} as Example 4.5.21. Given two positive integers $r,s$, consider the group given in the variety of all groups of nilpotency class 2 by the presentation
\[
H_{rs}:=\langle a,b \mid a^{2^r}=b^{2^s}=(a^{2^r-1}b^{2^s-1}ab)^2=1 \rangle.
\]
Then the quasivarieties generated by each of the following groups: the 4-element cyclic group $C_4$, the 8-element cyclic group $C_{8}$, $H_{22}$, $H_{33}$, $H_{23}$, and $H_{22}\times C_8$, form a 6-element non-modular sublattice in $L_q(\mathbf{G}_{8})$; see Figure~\ref{fig:hexagon}.
\end{proof}

\begin{figure}[htb]
\begin{center}
\unitlength=1mm
\begin{picture}(35,35)(0,5)
\put(10,10){\line(2,-1){10}}
\put(20,5){\line(4,5){12}}
\put(20,35){\line(4,-5){12}}
\put(10,10){\line(0,1){10}}
\put(10,20){\line(0,1){10}}
\put(10,30){\line(2,1){10}}
\put(10,10){\circle*{1}}
\put(20,35){\circle*{1}}
\put(10,20){\circle*{1}}
\put(10,30){\circle*{1}}
\put(20,5){\circle*{1}}
\put(32,20){\circle*{1}}
\put(21,3){$\qvar\{C_4\}$}
\put(-4,8){$\qvar\{C_{8}\}$}
\put(-5,20.5){$\qvar\{H_{33}\}$}
\put(-5,30.5){$\qvar\{H_{23}\}$}
\put(33,17){$\qvar\{H_{22}\}$}
\put(21,36){$\qvar\{H_{22}\times C_8\}$}
\end{picture}
\caption{A 6-element non-modular sublattice in the lattice $L_q(\mathbf{G}_{8})$}\label{fig:hexagon}
\end{center}
\end{figure}
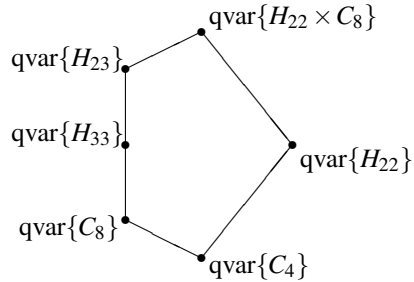

\section{Comparison between our approach and a specialization of Jackson's}

This section, unlike the rest of the paper, presupposes a somewhat deeper background in universal algebra.

As mentioned in the introduction, Jackson \cite{Jackson08} developed a very general meth\-od for translating universal Horn logic of partial algebras into equational logic of total algebras. In Sections \ref{subsec:UH}--\ref{subsec:groups}, we reproduce some constructions and facts from~\cite{Jackson08}, restricting them to total algebras. These restrictions suffice for our present purpose of clarifying which of our results could have been deduced from appropriate specializations of Jackson's and which could not; this is done in Section~\ref{subsec:comparison}.

We follow the terminology and notation of~\cite{Jackson08}, up to a few inessential details.

\subsection{Universal Horn formulas and universal Horn classes}
\label{subsec:UH}
A \emph{signature} $\mathrsfs{F}$ is a collection of symbols together with a map $\mathrsfs{F}\to\{0,1,\dots\}$ that assigns to each symbol its \emph{arity}, a non-negative integer. An \emph{$\mathrsfs{F}$-algebra} is a pair $(A;\,F)$, where $A$ is a set (the \emph{universe}) and $F$ is a collection of operations on $A$ which are in one-to-one correspondence with the symbols of $\mathrsfs{F}$ such that each $n$-ary operation in $F$ corresponds to a symbol of arity $n$. We use the same notation for an operation in $F$ and the corresponding symbol in $\mathrsfs{F}$, and refer to the elements of $\mathrsfs{F}$ as the \emph{fundamental} operations of~$(A;\,F)$.

$\mathrsfs{F}$-\emph{terms} are built from variables and the fundamental operations in the usual way. When we write $t(x,y)$ for a term, we assume that the variables $x$ and $y$ appear explicitly in the construction of the term.

\emph{Universal Horn formulas} (UH formulas, for short) of signature $\mathrsfs{F}$ are universally quantified formulas of one of the two forms:
\begin{gather}
\label{eq:Horn1st}
{u}_1=  {v}_1\ \&\ {u}_2=  {v}_2\ \&\ \cdots\ \&\ {u}_m=  {v}_m\longrightarrow {u}=  {v},\\
\label{eq:Horn2nd}
\neg({u}_1=  {v}_1) \lor \neg({u}_2=  {v}_2) \lor \cdots \lor \neg({u}_m={v}_m),
\end{gather}
where  ${u}_1,{v}_1,{u}_2,{v}_2,\dots,{u}_m,{v}_m,{u},{v}$ are $\mathrsfs{F}$-terms. UH formulas of the form \eqref{eq:Horn1st} are nothing but quasi-identities.

Satisfaction of UH formulas by an $\mathrsfs{F}$-algebra is defined in the usual manner. The class of all $\mathrsfs{F}$-algebras that satisfy all UH formulas from a given set $\Sigma$ is called the \emph{universal Horn class} (UH class, for short) axiomatized by $\Sigma$. Note that no UH formula of the form \eqref{eq:Horn2nd} can be satisfied by the one-element $\mathrsfs{F}$-algebra. Therefore, UH classes containing the one-element $\mathrsfs{F}$-algebra are quasivarieties.

The UH classes of $\mathrsfs{F}$-algebras form a complete lattice under class inclusion, provided the $\mathrsfs{F}$-algebra with empty universe is allowed, which we will assume in this section.

\subsection{Pointed semidiscriminator extensions} If $\mA=(A;\,F)$ is an $\mathrsfs{F}$-algebra and $\rhd\notin\mathrsfs{F}$, then $\mA^\rhd:=(A;\,F\cup\{\rhd\})$ where the operation $\rhd$ is defined by $a\rhd b:=b$ for all $a,b\in A$.

Let $\mA:=(A;\,F)$ be an $\mathrsfs{F}$-algebra and let $\infty\notin A$ and $\wedge\notin\mathrsfs{F}$. The \emph{flat extension} of $\mA$ is the algebra $\mA^\flat:=(A\cup\{\infty\};\,F\cup\{\wedge\})$ where the operation $\wedge$ is defined by
\begin{equation*}
\label{eq:flatgen}
a\wedge b:=\begin{cases}
\infty&\text{if } a\ne b,\\
a&\text{if } a=b,
\end{cases}
\quad\text{for all}\ a,b\in A\cup\{\infty\},
\end{equation*}
and each operation in $\mathrsfs{F}$ returns the same values as in $\mA$ whenever all of its arguments come from the set $A$, and returns $\infty$ if at least one of its arguments is $\infty$.

The \emph{pointed semidiscriminator extension} $\ps(\mA)$ of an $\mathrsfs{F}$-algebra $\mA$ is defined as $(\mA^\rhd)^\flat$, that is, as the flat extension of the algebra $\mA^\rhd$. Thus, if $\mA=(A;\,F)$, then the universe of the algebra $\ps(\mA)$ is $A\cup\{\infty\}$ and the signature of $\ps(\mA)$ is $\mathrsfs{F}\cup\{\rhd,\wedge\}$.

Notice that by the definition of a flat extension, the operation $\rhd$ on $\ps(\mA)$ returns $\infty$ if at least one of its arguments is $\infty$; otherwise, it returns its second argument. Therefore,
\[
(a\wedge b)\rhd c=\begin{cases}
c &\text{if } a=b\ne\infty,\\
\infty&\text{otherwise,}
\end{cases}
\quad\text{for all}\ a,b,c\in A\cup\{\infty\}.
\]
In the terminology of~\cite{Jackson08}, this means that the term $(x\wedge y)\rhd z$ is a \emph{pointed semidiscriminator term} in the algebra $\ps(\mA)$. The presence of such a term underlies many ``positive'' features of pointed semidiscriminator extensions and makes the construction $\ps(\mA)$ very useful.

\begin{remark}
\label{rem:flat}
If an $\mathrsfs{F}$-algebra $\mA$ satisfies the identity $t(x,y)=y$ for some $\mathrsfs{F}$-term $t(x, y)$, then the pointed semidiscriminator extension $\ps(\mA)$ is term equivalent to the flat extension $\mA^\flat$. In particular, the latter extension admits $t(x\wedge y,z)$ as a pointed semidiscriminator term. A simple example relevant to the later discussion in this section is given by any group of exponent dividing $n$, for which one may take $t(x,y):= x^n y$.
\end{remark}

\subsection{{The maps $\mathbf{H}\mapsto\mathbf{H}^\flat$ and $\mathbf{H}\mapsto\ps(\mathbf{H}$)}{}}
For a UH class $\mathbf{H}$ of $\mathrsfs{F}$-algebras, let
\begin{align*}
\mathbf{H}^\flat&:=\var\{\mA^\flat\mid \mA\in\mathbf{H}\},\\
\ps(\mathbf{H})&:=\var\{\ps(\mA)\mid \mA\in\mathbf{H}\}.
\end{align*}

The following fact is a restriction of Theorem 5.3 in \cite{Jackson08} to total algebras. The theorem, in turn, is deduced from Theorem 3.3, the key result of~\cite{Jackson08}.
\begin{fact}
\label{fact:subirr}
For every UH class\/ $\mathbf{H}$, the subdirectly irreducible members of the variety $\ps(\mathbf{H})$ are simple and are, up to isomorphism, precisely the pointed semidiscriminator extensions of members of\/ $\mathbf{H}$.
\end{fact}

From this, one readily derives the next fact, which is a restriction of \cite[Corollary 5.4]{Jackson08}.

\begin{fact}
\label{fact:embedding}
The map\/ $\mathbf{H}\mapsto\ps(\mathbf{H})$ is a lattice embedding from the lattice of all UH classes of $\mathrsfs{F}$-algebras into the lattice of varieties of\/ $(\mathrsfs{F}\cup\{\rhd,\wedge\})$-algebras.
\end{fact}

By Remark~\ref{rem:flat}, if a UH class $\mathbf{H}$ satisfies an identity of the form $t(x,y)=y$ for some term $t(x, y)$, one can use flat extensions in place of pointed semidiscriminator extensions in Fact~\ref{fact:subirr}. From this, one can obtain the following analogue of Fact~\ref{fact:embedding} (not explicitly stated in \cite{Jackson08}):

\begin{fact}
\label{fact:flatembedding}
For each $\mathrsfs{F}$-term $t(x,y)$, the map\/ $\mathbf{H}\mapsto\mathbf{H}^\flat$ is a lattice embedding from the lattice of all UH classes of $\mathrsfs{F}$-algebras satisfying the identity $t(x,y)=y$ into the lattice of varieties of\/ $(\mathrsfs{F}\cup\{\wedge\})$-algebras.
\end{fact}

In \cite[p.~115]{Jackson08}, Jackson presents a transformation that, given a UH formula $\Phi$ of signature $\mathrsfs{F}$, produces an identity $\Phi^\rhd$ of signature $\mathrsfs{F}\cup\{\rhd,\wedge\}$. The key property of this transformation is that an $\mathrsfs{F}$-algebra $\mA$ satisfies $\Phi$ if and only if its pointed semidiscriminator extension $\ps(\mA)$ satisfies $\Phi^\rhd$ \cite[Lemma 5.9]{Jackson08}. Applying the transformation $\Phi\mapsto \Phi^\rhd$ to each formula in a (finite) basis of UH formulas of a UH class $\mathbf{H}$ leads to a (finite) identity basis of the variety $\ps(\mathbf{H})$ relative to the variety generated by all pointed semidiscriminator algebras \cite[Theorem 5.12]{Jackson08}. If signature $\mathrsfs{F}$ is finite, then the latter variety is \fb, whence the variety $\ps(\mathbf{H})$ has a finite identity basis whenever the UH class $\mathbf{H}$ has a finite basis of UH formulas. For UH classes $\mathbf{H}$ satisfying an identity of the form $t(x,y)=y$ for some term $t(x,y)$, the same procedure yields a (finite) identity basis of the variety $\mathbf{H}^\flat$.

\subsection{Flat extensions of groups as naturally semilattice ordered Clifford semigroups}
\label{subsec:groups}
Sections 7.7 and 7.8 in \cite{Jackson08} and Section 2 in \cite{RJZL23} specialize some general results of \cite{Jackson08} to the case of groups. Since every group has the one-element group as a subgroup, every UH class of groups is a quasivariety. Let $\mathbf{G}$ stand for the variety of all groups considered as $\{\cdot,{}^{-1}\}$-algebras. The identity $xx^{-1}y=y$ holds in $\mathbf{G}$, so Fact~\ref{fact:flatembedding} applies, yielding  a lattice isomorphism from the lattice $L_q(\mathbf{G})$ of all subquasivarieties of $\mathbf{G}$ onto the lattice of all nontrivial subvarieties of the variety $\mathbf{G}^\flat$. A similar application of Fact~\ref{fact:flatembedding}, based on the presence of the identity $x^ny=y$ in any group of exponent dividing $n$, yields the following:

\begin{fact}
\label{fact:expnembedding}
The map\/ $\mathbf{H}\mapsto\mathbf{H}^\flat$ is a lattice isomorphism between the lattice $L_q(\mathbf{G}_n)$ of all quasivarieties of groups of exponent dividing $n$ and the lattice of nontrivial subvarieties of the variety $\mathbf{G}_n^\flat$, and this map preserves finite axiomatizability and nonfinite axiomatizability.
\end{fact}

Fact~\ref{fact:expnembedding} appears as Theorem 2.1 in \cite{RJZL23}, stated without proof. The authors merely note that the result is a restriction of \cite[Corollary 5.4, Theorem 5.12]{Jackson08} and comment on this in a footnote on p. 67. The above discussion essentially elaborates on that footnote.

Section 7.8 in \cite{Jackson08} discusses $\mathbf{G}^\flat$ as a variety of $\{\cdot,{}^{-1},\wedge\}$-algebras.

Recall that every Clifford semigroup $(S;\,\cdot)$ is a union of its subgroups. It is known (and easy to verify) that the inverse of an element $s \in S$ in any subgroup containing $s$ does not depend on the choice of the subgroup. This yields a well-defined unary operation $s\mapsto s^{-1}$ on $S$. Equipped with this unary operation, Clifford semigroups form a well-understood variety of $\{\cdot,{}^{-1}\}$-algebras.

A \emph{naturally semilattice ordered Clifford semigroup} is an algebra $(S;\,\cdot,{}^{-1},\wedge)$ such that $(S;\,\cdot,{}^{-1})$ is a Clifford semigroup, $(S;\,\wedge)$ is a semilattice, $\cdot$ distributes over $\wedge$ on the left and on the right, and the identity
\begin{equation}\label{eq:natorder}
x\wedge y=x(x\wedge y)^{-1}(x\wedge y)
\end{equation}
holds. To explain the name ``naturally semilattice ordered'', recall that in every Clifford semigroup $(S;\,\cdot,{}^{-1})$, the relation
\[
\le_{\mathrm{nat}}:= \{(a,b)\in S\times S\mid a=ba^{-1}a\}
\]
is a partial order, known as the \emph{natural partial order}. The identity \eqref{eq:natorder} expresses the fact that $a\wedge b$ is the greatest lower bound of $a$ and $b$ with respect to $\le_{\mathrm{nat}}$ in $(S;\,\cdot,{}^{-1},\wedge)$; see \cite[Proposition 3.1]{JackStok03}. Thus, the semilattice $(S;\,\wedge)$ also constitutes a semilattice under the natural partial order.

\begin{fact}[{\!\mdseries\cite[Theorem 7.5]{Jackson08}}]
\label{fact:natorder}
$\mathbf{G}^\flat$ as a variety of $\{\cdot,{}^{-1},\wedge\}$-algebras is the variety of all naturally semilattice ordered Clifford semigroups.
\end{fact}

The proof of Fact~\ref{fact:natorder} in \cite{Jackson08} is syntactic and uses the transformation $\Phi\mapsto\Phi^\rhd$ mentioned at the end of the preceding section.

No characterization of the variety $\mathbf{G}_n^\flat$ is available in \cite{Jackson08} or in \cite{RJZL23}. If one follows arguments from the proof of Fact~\ref{fact:natorder}, using the term $x^ny$ instead of $xx^{-1}y$ and writing $+$ in place of $\wedge$, one can identify $\mathbf{G}_n^\flat$ with the variety $\mathbf{N}_n$ of all \ais{}s from $\mathbf{Sr}_n$ that satisfy identities \eqref{eq:center} and
\begin{equation}\label{eq:natordern}
x+y=x(x+y)^n.
\end{equation}

\subsection{Comparison}
\label{subsec:comparison}

Our main result is Theorem~\ref{thm:main}. It:
\begin{itemize}
  \item shows that the lattice $L_q(\mathbf{G}_n)$ of all quasivarieties of groups of exponent dividing $n$ embeds into the lattice $L(\mathbf{Sr}_n)$ as an interval and that the embedding preserves finite axiomatizability and nonfinite axiomatizability;
  \item identifies the top element $\mathbf{G}_n^\flat$ of this interval as the variety $\mathbf{M}_n$ defined within $\mathbf{Sr}_n$ by identity \eqref{new3}.
\end{itemize}

The first claim coincides with Fact~\ref{fact:expnembedding} and is therefore not new, but our proof differs substantially from those in \cite{Jackson08,RJZL23}. In particular, it is carried out entirely within the framework of \ais{}s and relies only on basic algebraic tools.

The fact that the variety $\mathbf{G}_n^\flat$ can be identified as $\mathbf{M}_n$ is new, to the best of our knowledge. As discussed in Section~\ref{subsec:groups}, no characterization of $\mathbf{G}_n^\flat$ is given in  \cite{RJZL23}, and the characterization of $\mathbf{G}_n^\flat$ as the variety $\mathbf{N}_n$, which could have been extracted from arguments in \cite[ Section 7.8]{Jackson08}, differs from ours.

A posteriori, it is clear that $\mathbf{M}_n=\mathbf{N}_n$ as both varieties coincide with $\mathbf{G}_n^\flat$. This equality can also be established by direct computation with the defining identities of the varieties.

Indeed, showing that $\mathbf{N}_n\subseteq \mathbf{M}_n$ amounts to deriving identity \eqref{new3} from identities \eqref{eq:srn}, \eqref{eq:center}, and \eqref{eq:natordern}. Substituting $x^n$ for $x$ and $y^n$ for $y$ in identity \eqref{eq:natordern}, we obtain
\begin{align*}
x^n+y^n&=x^n(x^n+y^n)^n &&\\
       &=x^n(x^n+y^n)&&\text{since $(x^n+y^n)^2=x^n+y^n$ in view of \eqref{eq:srn} and Lemma~\ref{lem:e(s)}.}
\end{align*}
Swapping $x$ and $y$ and using the commutativity of addition, we similarly deduce
\[
x^n+y^n=y^n(x^n+y^n).
\]
Multiplying both sides of the two identities just obtained yields
\begin{multline}\label{eq:product}
(x^n+y^n)^2=x^n(x^n+y^n)y^n(x^n+y^n)\\
=(x^n)^2y^nx^n+(x^n)^2(y^n)^2+x^n(y^n)^2x^n+x^n(y^n)^3.
\end{multline}
In view of \eqref{eq:srn} and Lemma~\ref{lem:e(s)}, we have $(x^n+y^n)^2=x^n+y^n$, $(x^n)^2=x^n$, and $(y^n)^2=y^n$, while \eqref{eq:center} implies $y^nx^n=x^ny^n$. Hence, the left-hand side of \eqref{eq:product} simplifies to $x^n+y^n$ and each summand in the right-hand side of \eqref{eq:product} simplifies to $x^ny^n$. Using the idempotency of addition, we obtain $x^n+y^n=x^ny^n$, that is, the required identity \eqref{new3}.

Conversely, to show that $\mathbf{M}_n\subseteq \mathbf{N}_n$, one has to derive identities \eqref{eq:center} and \eqref{eq:natordern} from identity \eqref{new3}. That \eqref{eq:center} follows from \eqref{new3} was shown in Lemma~\ref{lem:3iden}. Using Lemma~\ref{lem:3iden}, we also obtain
\[
x+y=y+x\stackrel{\eqref{eq:summing}}{=}\mathrm{M}(yx^{n-1})x.
\]
To lighten notation, we let $e:=\mathrm{M}(yx^{n-1})$ so that the identity just established becomes $x+y=ex$. By Lemma~\ref{lem:m}, we have $e^2=e$. Multiplying both sides of $x+y=ex$ by $e$ on the left yields
\begin{equation}\label{eq:e(x+y)}
e(x+y)=e^2x=ex=x+y.
\end{equation}
By \eqref{eq:center}, we have $ex=xe$ whence $x+y=xe$. Now we compute
\begin{align*}
x+y&=(x+y)^{n+1}&&\text{by \eqref{eq:srn}}\\
   &=(x+y)(x+y)^n&&\\
   &=xe(x+y)^n&&\text{since $x+y=xe$}\\
   &=x(x+y)^n&&\text{by \eqref{eq:e(x+y)}},
\end{align*}
thus obtaining the required identity \eqref{eq:natordern}.

Returning to the relationship between the results of \cite{Jackson08} and those of the present paper, we note that \emph{after} the equality $\mathbf{M}_n=\mathbf{N}_n=\mathbf{G}_n^\flat$ is established, our classification of the subdirectly irreducible members of $\mathbf{M}_n$ (Proposition~\ref{prop:si}) becomes a special instance of Fact~\ref{fact:subirr}. We believe, however, that our direct proofs in Section~\ref{sec:mn}, employing only elementary \ais{} tools, still offer independent value and have the potential to be extended to other subvarieties of the variety $\mathbf{Sr}_n$.

\subsection*{Acknowledgements}
The authors are greatly indebted to Marcel Jackson for stimulating discussions over the years and to Alexander Budkin for providing up-to-date information on lattices of group quasivarieties.

The authors are grateful to the anonymous referee for their detailed feedback and several helpful suggestions, especially for spotting an erroneous step in the argument originally used to justify Theorem~\ref{thm:main}.

\end{document}